\documentclass[10pt]{amsart}
\usepackage{enumerate,amsmath,amssymb,latexsym,
amsfonts, amsthm, amscd, MnSymbol}

\theoremstyle{plain}

\newtheorem{theorem}{Theorem}

\numberwithin{equation}{section}

\begin{document}

\title {Gauss $q$-ed from Heine cubed}

\date{}

\author[P.L. Robinson]{P.L. Robinson}

\address{Department of Mathematics \\ University of Florida \\ Gainesville FL 32611  USA }

\email[]{paulr@ufl.edu}

 \keywords{}

\begin{abstract}

We consider $q$-analytic derivations of the $q$-Gauss summation formula for a $\, _2\phi _1$ that respect the symmetry in its upper parameters. 

\end{abstract}

\maketitle

Among the many original results in the pioneering paper [1847] of Heine is his $q$-series analogue for the Gauss summation of a hypergeometric series. Recall that the $q$-series analogue $\, _2\phi _1$ of the hypergeometric series $_2F_1$ is defined by 
\setlength\arraycolsep{1pt}
$$\, _2\phi _1 \left(\begin{matrix}a& &b& \\&c&
\end{matrix}; q ; z\right) = \sum_{n=0}^{\infty} \frac{(a; q)_n (b; q)_n}{(c; q)_n} \frac{z^n}{(q; q)_n}\,.$$
Here and in what follows, it is assumed that $|q| < 1$ and that the lower parameter $c$ is not equal to $q^{- N}$ for any integer $N \geqslant 0$. In these terms, the $q$-Gauss summation formula of Heine asserts that if $|c| < |ab|$ then 
$$\, _2\phi _1 \left(\begin{matrix}a& &b& \\&c&
\end{matrix}; q; \frac{c}{ab}\right) = \frac{(c/a; q)_{\infty} (c/b; q)_{\infty}}{(c; q)_{\infty} (c/ab; q)_{\infty}}\, .$$
This appears as Formula 80 on page 307 of [1847]; naturally, it is there expressed after the fashion of the time. 

\medbreak 

The traditional proof of the $q$-Gauss summation formula derives it as a consequence of the fundamental Heine transformation for $_2\phi_1$: this asserts that if $z$ and $b$ lie (along with $q$) in the open unit disc then 
\setlength\arraycolsep{1pt}
$$\, _2\phi _1 \left(\begin{matrix}a& &b& \\&c&
\end{matrix}; q ; z\right) = \frac{(a z; q)_{\infty} (b; q)_{\infty}}{(z; q)_{\infty} (c; q)_{\infty}} \, \, _2\phi _1 \left(\begin{matrix}c/b& &z& \\&a z&
\end{matrix}; q ; b\right).$$
This transformation of Heine is central to [1847]: in fact, it appears on page 306 as Formula 79 in the form 
$$\frac{(c; q)_{\infty}}{(b; q)_{\infty}} \, \, _2\phi _1 \left(\begin{matrix}a& &b& \\&c&
\end{matrix}; q ; z\right) = \frac{(a z; q)_{\infty}}{(z; q)_{\infty}} \, \, _2\phi _1 \left(\begin{matrix}c/b& &z& \\&a z&
\end{matrix}; q ; b\right).$$
Heine establishes the $q$-Gauss summation formula by evaluating his transformation formula at $z = c/ab$ and invoking the $q$-binomial series, the latter of which appears as Formula 74 on page 303 of [1847]. 

\medbreak 

This original proof by Heine is one of the few known proofs of the $q$-Gauss summation formula. Ramanujan sketched an approach based on the $q$-binomial series in one of the manuscripts included in [1988] along with his `Lost Notebook'. Bailey [1935] offered a less elementary approach based on the $q$-Dougall identity of Jackson for an $\, _8\phi _7$. None of these $q$-analytic derivations of the $q$-Gauss summation formula respects the symmetry in its upper parameters $a$ and $b$: this claim is at once clear in the case of Ramanujan's approach; in the case of Bailey's approach, the validity of the claim becomes evident upon closer inspection. Our aim in this short note is to consider $q$-analytic derivations of the $q$-Gauss summation formula that fully respect its symmetry. Section 1 reviews the Heine transformation formula and its group of symmetries; Section 2 addresses the task of developing a symmetric derivation of the $q$-Gauss summation formula. 

\medbreak 

\section*{Heine symmetries}

\medbreak 

We begin with a mild notational simplification. As $\, _2\phi _1$ is the only $q$-hypergeometric series that we shall consider, we drop `2' and `1' as subscripts, transferring the base $q$ to one of the vacancies. Thus, we define 
\setlength\arraycolsep{1pt}
$$\phi _q \left(\begin{matrix}a& &b& \\&c&
\end{matrix}; z\right) = \sum_{n=0}^{\infty} \frac{(a; q)_n (b; q)_n}{(c; q)_n} \frac{z^n}{(q; q)_n}$$
assuming throughout that $q$ lies in the open unit disc and that $c q^N \neq 1$ for each integer $N \geqslant 0$.

\medbreak 

With this understanding, the fundamental transformation of Heine [1847] asserts that if also $b$ and $z$ lie in the open unit disc then 
\setlength\arraycolsep{1pt}
$$\phi _q \left(\begin{matrix}a& &b& \\&c&
\end{matrix}; z\right) = \frac{(b; q)_{\infty} (a z; q)_{\infty}}{(c; q)_{\infty} (z; q)_{\infty}} \, \phi _q \left(\begin{matrix}c/b& &z& \\&a z&
\end{matrix}; b\right).$$
This is the way in which the transformation is presented according to custom. Our purposes will be better served by exploiting the symmetry of $\phi_q$ in its upper parameters: thus, we shall cast the Heine transformation as 
$$\phi _q \left(\begin{matrix}a& &b& \\&c&
\end{matrix}; z\right) = \frac{(b; q)_{\infty} (a z; q)_{\infty}}{(c; q)_{\infty} (z; q)_{\infty}} \, \phi _q \left(\begin{matrix}z& &c/b& \\&a z&
\end{matrix}; b\right).$$
Of course, the case $b = 0$ calls for separate treatment, which we leave to the reader. 

\medbreak 

For organizational reasons, we introduce the parameter-variable space 
$$\mathbb{S} = \Big\{ \left(\begin{matrix}a& &b& \\&c&
\end{matrix}; z\right) : a, b, c, z \in \mathbb{C} \setminus \{0\} \Big\}.$$ 
On $\mathbb{S}$ we define the operator 
$$U : \mathbb{S} \to \mathbb{S} : \left(\begin{matrix}a& &b& \\&c&
\end{matrix}; z\right) \mapsto \left(\begin{matrix}b& &a& \\&c&
\end{matrix}; z\right)$$ 
that interchanges the upper parameters. Symmetry of $\phi_q$ in its upper parameters now reads 
$$\phi_q \circ U = \phi_q.$$ 
On $\mathbb{S}$ we also introduce the operator 
$$H: \mathbb{S} \to \mathbb{S} : \left(\begin{matrix}a& &b& \\&c&
\end{matrix}; z\right) \mapsto \left(\begin{matrix}z& &c/b& \\&az&
\end{matrix}; b\right)$$ 
along with the `automorphy factor' 
$$h : \mathbb{S} \to \mathbb{C} : \left(\begin{matrix}a& &b& \\&c&
\end{matrix}; z\right) \mapsto \frac{(c; q)_{\infty} (z; q)_{\infty}}{(b; q)_{\infty} (a z; q)_{\infty}}\, .$$ 
The fundamental Heine transformation now asserts that if 
$$s = \left(\begin{matrix}a& &b& \\&c&
\end{matrix}; z\right) \in \mathbb{S}$$
then 
$$\phi_q (H s) = h(s) \phi_q (s)$$ 
where it is assumed that $b$ and $z$ lie in the open unit disc. Note that the composite $U H$ corresponds to the customary way of presenting the Heine transformation. 

\medbreak 

\begin{theorem} \label{group}
The operators $U$ and $UH$ have period two, while $H$ has period six. 
\end{theorem} 

\begin{proof} 
That $U$ and $UH$ have period two is a matter of simple verification. The verification that $H$ has period six is likewise simple; for future reference, we record the iterates 
$$H^2 \left(\begin{matrix}a& &b& \\&c&
\end{matrix}; z\right) = \left(\begin{matrix}b& &abz/c& \\&bz&
\end{matrix}; c/b\right)$$
$$H^3 \left(\begin{matrix}a& &b& \\&c&
\end{matrix}; z\right) = \left(\begin{matrix}c/b& &c/a& \\&c&
\end{matrix}; abz/c\right)$$
and note that $H^3$ is an involution. 
\end{proof} 

\medbreak 

It follows at once that the group generated by $U$ and the Heine operator $H$ is dihedral of order twelve, the involution $H^3$ generating its centre. The value of this symmetry group was noted by Rogers, who exploited it in his study [1893] of the Heine transformation. 

\medbreak 

The iterates of the Heine operator have readily calculable effects on $\phi_q$. If 
$$s = \left(\begin{matrix}a& &b& \\&c&
\end{matrix}; z\right) \in \mathbb{S}$$
then 
$$\phi_q (H^2 s) = \phi_q (H H s) = h(H s) \phi_q (H s) = h(H s) h(s) \phi_q (s)$$ 
where 
$$h(s) = \frac{(c; q)_{\infty} (z; q)_{\infty}}{(b; q)_{\infty} (a z; q)_{\infty}}$$
and 
$$h(H s) = \frac{(az; q)_{\infty} (b; q)_{\infty}}{(c/b; q)_{\infty} (b z; q)_{\infty}}$$
so that 
$$\phi_q \left(\begin{matrix}b& &abz/c& \\&bz&
\end{matrix}; c/b\right) = \frac{(c; q)_{\infty} (z; q)_{\infty}}{(c/b; q)_{\infty} (b z; q)_{\infty}} \, \phi_q \left(\begin{matrix}a& &b& \\&c&
\end{matrix}; z\right)$$
in view of the formula for $H^2$ recorded in the proof of Theorem \ref{group}. Here $c/b$ and $z$ lie in the open unit disc, the intermediate requirement $|b| < 1$ being removed by analytic continuation. The effect of the cube $H^3$ we record as follows. 

\medbreak 

\begin{theorem} \label{cube}
If $z$ and $abz/c$ lie in the open unit disc then 
$$(z; q)_{\infty} \, \phi_q \left(\begin{matrix}a& &b& \\&c&
\end{matrix}; z\right) = (abz/c; q)_{\infty} \, \phi_q \left(\begin{matrix}c/b& &c/a& \\&c&
\end{matrix}; abz/c\right).$$ 
\end{theorem} 

\begin{proof} 
Argue essentially as above, using the formulae recorded in the proof of Theorem \ref{group}. Again, analytic continuation authorizes the lifting of catalytic intermediate requirements. 
\end{proof} 

\medbreak 

The cube of the Heine transformation is sometimes called the $q$-Euler transformation. Note that it is symmetric in the upper parameters $a$ and $b$; in this, it agrees with the $q$-Gauss summation formula. 

\medbreak 

\section*{Gauss symmetrically}

\medbreak 

We now turn to the $q$-Gauss summation formula, aiming at a derivation that fully respects its symmetry in the upper parameters. 

\medbreak 

For purposes of comparison we first recall the traditional derivation, presented in [1847]. This derivation rests on the $q$-binomial series, according to which 
$$\sum_{n = 0}^{\infty} \frac{(u; q)_n}{(q; q)_n} z^n = \frac{(u z; q)_{\infty}}{(z; q)_{\infty}}$$
if $z$ (along with $q$ as usual) lies in the open unit disc. In the Heine transformation 
$$\phi _q \left(\begin{matrix}a& &b& \\&c&
\end{matrix}; z\right) = \frac{(b; q)_{\infty} (a z; q)_{\infty}}{(c; q)_{\infty} (z; q)_{\infty}} \, \phi _q \left(\begin{matrix}z& &c/b& \\&a z&
\end{matrix}; b\right)$$
put $z = c/ab$: the right-hand side then becomes 
$$\frac{(b; q)_{\infty} (c/b; q)_{\infty}}{(c; q)_{\infty} (c/ab; q)_{\infty}} \, \phi _q \left(\begin{matrix}c/ab& &c/b& \\&c/b&
\end{matrix}; b\right)$$
in which the $q$-binomial series evaluates the $\phi_q$ factor as 
$$\frac{((c/ab)b; q)_{\infty}}{(b; q)_{\infty}} = \frac{(c/a; q)_{\infty}}{(b; q)_{\infty}}$$
whence follows the $q$-Gauss summation formula 
$$\phi _q \left(\begin{matrix}a& &b& \\&c&
\end{matrix}; \frac{c}{ab}\right) = \frac{(c/a; q)_{\infty} (c/b; q)_{\infty}}{(c; q)_{\infty} (c/ab; q)_{\infty}}\, .$$
Alternatively, the square of the Heine transformation offers a more interesting derivation. In 
$$\phi_q \left(\begin{matrix}a& &b& \\&c&
\end{matrix}; z\right) = \frac{(c/b; q)_{\infty} (b z; q)_{\infty}}{(c; q)_{\infty} (z; q)_{\infty}} \, \phi_q \left(\begin{matrix}b& &abz/c& \\&bz&
\end{matrix}; c/b\right)$$
again put $z = c/ab$: the $q$-Gauss summation formula follows at once, because 
$$\phi_q  \left(\begin{matrix}b& &1& \\&c/a&
\end{matrix}; c/b\right) = 1$$
by virtue of the fact that $(1; q)_n$ is the Kronecker delta $\delta_{0 n}$. 

\medbreak 

Of course, neither of these derivations treats the upper parameters even-handedly. The cube of the Heine transformation shows more promise, as it is symmetric in the upper parameters. In order to apply `Heine cubed' it is convenient to record the following elementary fact, in the statement of which we continue to suppress the underlying assumption $|q| < 1$. 

\medbreak 

\begin{theorem} \label{Abel}
$$\lim_{z \uparrow 1}  \, (z; q)_{\infty} \, \phi_q \left(\begin{matrix}a& &b& \\&c&
\end{matrix}; z\right) = \frac{(a; q)_{\infty} (b; q)_{\infty}}{(c; q)_{\infty}}\, .$$ 
\end{theorem} 

\begin{proof} 
Simply apply the Abel limit theorem to 
$$ (z; q)_{\infty} \, \phi_q \left(\begin{matrix}a& &b& \\&c&
\end{matrix}; z\right) = (z q; q)_{\infty} \cdot (1 - z) \sum_{n = 0}^{\infty} \frac{(a; q)_n (b; q)_n}{(q; q)_n (c; q)_n} \, z^n .$$
\end{proof} 

\medbreak 

Here, the limit may instead be taken within a Stolz sector, as usual. Of course, this theorem hints at the notion of Abel summability for possibly divergent $q$-series. 

\medbreak 

After this preparation, the $q$-Gauss summation formula may be derived as follows. 

\medbreak 

\begin{theorem} \label{Gauss}
If $|q| < 1$ and $|c| < | a b |$ then 
$$\phi _q \left(\begin{matrix}a& &b& \\&c&
\end{matrix}; \frac{c}{ab}\right) = \frac{(c/a; q)_{\infty} (c/b; q)_{\infty}}{(c; q)_{\infty} (c/ab; q)_{\infty}}\, .$$
\end{theorem} 

\begin{proof} 
We start from the $q$-Euler transformation of Theorem \ref{cube} in the equivalent form 
$$(cz/ab; q)_{\infty} \, \phi_q \left(\begin{matrix}a& &b& \\&c&
\end{matrix}; cz/ab\right) = (z; q)_{\infty} \, \phi_q \left(\begin{matrix}c/b& &c/a& \\&c&
\end{matrix}; z\right)$$
where $q$, $z$ and $cz/ab$ lie in the open unit disc. Now let $z \uparrow 1$: the left-hand side plainly has limit 
$$(c/ab; q)_{\infty} \, \phi_q \left(\begin{matrix}a& &b& \\&c&
\end{matrix}; c/ab\right)$$
while the right-hand side has limit 
$$\frac{(c/b; q)_{\infty} (c/a; q)_{\infty}}{(c; q)_{\infty}}$$
on account of Theorem \ref{Abel}. 
\end{proof} 

\medbreak 

A purist will rightly object that the `symmetry' of this approach to the $q$-Gauss summation formula is undermined somewhat by the circumstance that, whereas the $q$-Euler transformation is symmetric in $a$ and $b$, the Heine transformation from which it is derived as the cube is not. This objection can be countered by an alternative derivation of the $q$-Euler transformation that handles the upper parameters $a$ and $b$ on equal terms, avoiding use of the Heine transformation. Such a derivation is made possible by the $q$-Pfaff-Saalsch\"utz identity. In his eight-page eighth chapter, Bailey [1935] deduces the $q$-Euler transformation and the $q$-Gauss summation formula from the $q$-Pfaff-Saalsch\"utz identity. However, his derivation of the $q$-Pfaff-Saalsch\"utz identity itself is perhaps not quite elementary, being based on the $q$-Dougall identity of Jackson for an $\, _8\phi _7$; more to the point, his proof of the $q$-Dougall identity involves a breaking of symmetry, in that one of the upper parameters (there called $b$ and $c$) is singled out for special handling. The attractive recent introduction to $q$-analysis [2020] by Johnson places us within reach of thoroughly symmetric $q$-analytic derivations of the $q$-Gauss summation formula. In Section 5.7 Johnson passes from the $q$-Euler transformation to the $q$-Pfaff-Saalsch\"utz identity by a careful comparison of coefficients. On the one hand, reconstitution effects a passage in the opposite direction, yielding a symmetric derivation of the $q$-Euler transformation from the $q$-Pfaff-Saalsch\"utz identity; on the other hand, the $q$-Pfaff-Saalsch\"utz identity is a finite version of the $q$-Gauss summation formula. Last but not least, the final exercise in Section 5.7 establishes the $q$-Pfaff-Saalsch\"utz identity from first principles, while simultaneously treating the upper parameters symmetrically. 

\bigbreak

\begin{center} 
{\small R}{\footnotesize EFERENCES}
\end{center} 
\medbreak 

[1847] E. Heine, {\it Untersuchungen \"uber die Reihe}, Journal f\"ur die reine und angewandte Mathematik {\bf 34} 285-328. 

\medbreak 

[1893] L.J. Rogers, {\it On a Three-fold Symmetry in the Elements of Heine's series}, Proceedings of the London Mathematical Society {\bf 24} 171-179. 

\medbreak 

[1935] W.N. Bailey, {\it Generalized Hypergeometric Series}, Cambridge Tracts in Mathematics and Mathematical Physics {\bf 32}, Cambridge University Press. 

\medbreak 

[1988] S. Ramanujan, {\it The Lost Notebook and Other Unpublished Papers}, Narosa, New Delhi. 

\medbreak 

[2020] W.P. Johnson, {\it An Introduction to $q$-analysis}, American Mathematical Society, Providence, Rhode Island.

\end{document}